\documentclass{amsart}
\usepackage{etex}

\usepackage{amsfonts,amsmath,amssymb,amsthm,amscd,amsxtra}
\usepackage{enumerate,epsfig,verbatim}

\usepackage{mathptmx}

\usepackage{amsfonts,amsmath,amssymb,amsthm,amscd,amsxtra}
\usepackage{enumerate,verbatim}
\usepackage{centernot}
\usepackage{todonotes}

\usepackage[T1]{fontenc}
\usepackage[usenames,dvipsnames]{pstricks}
\usepackage[mathscr]{eucal}
\usepackage[notcite,notref]{}
\usepackage{tikz}
\usetikzlibrary{matrix,quotes}
\usepackage[notcite, notref]{}
\usepackage[usenames,dvipsnames]{pstricks}
\usepackage[mathscr]{eucal}
\usepackage{amsfonts,amsmath,amssymb,amsthm,amscd,amsxtra}
\usepackage{enumerate,verbatim}
\usepackage[all,2cell,ps]{xy}
\usepackage{mathptmx}
\usepackage[notcite, notref]{}
\usepackage[pagebackref]{hyperref}
\usepackage{mathtools}
\usepackage{nicefrac}
\usepackage{array}
\usepackage{xcolor}
\usepackage[all,2cell,ps]{xy}
\usepackage{amsfonts}
\usepackage[margin=1.4in]{geometry}
\usepackage{color}
\usepackage[notcite,notref]{}
\usepackage{url}
\usepackage{datetime}
\usepackage{mathtools}
\usepackage[all,2cell,ps]{xy}
\usepackage{amssymb}
\usepackage{amsmath}
\usepackage{amsfonts}
\usepackage{amsmath}
\usepackage{amsthm}
\usepackage{amssymb}
\usepackage{amscd}
\usepackage{amsfonts}
\usepackage{amsxtra}     
\usepackage{epsfig}
\usepackage{verbatim}
\usepackage[all]{xypic}

\SelectTips{cm}{}

\newcommand{\comm}[1]{{\color[rgb]{0.0, 0.5, 0.0} #1}}

\newcommand{\new}[1]{{\color{blue} #1}}

\def\latex/{{\protect\LaTeX}}
\def\latexe/{{\protect\LaTeXe}}
\def\amslatex/{{\protect\AmS-\protect\LaTeX}}
\def\tex/{{\protect\TeX}}
\def\amstex/{{\protect\AmS-\protect\TeX}}
\def\bibtex/{{Bib\protect\TeX}}
\def\makeindx/{\textit{MakeIndex}}

\usepackage{amsmath}
\usepackage{amsthm}
\usepackage{amssymb}
\usepackage{amscd}
\usepackage{amsxtra}     
\usepackage{epsfig}
\usepackage{verbatim}
\usepackage[all]{xypic}
\usepackage{enumerate}

\theoremstyle{plain} 

\newtheorem{thm}{Theorem}[section]

\newtheorem{prop}[thm]{Proposition}
\newtheorem{cor}[thm]{Corollary}

\theoremstyle{definition}

\newtheorem{chunk}[thm]{\hspace*{-1.065ex}\bf}

\newtheorem{lem}[thm]{Lemma}
\newtheorem{dfn}[thm]{Definition}

\newtheorem{eg}[thm]{Example}
\newtheorem{conj}[thm]{Conjecture}
\newtheorem{ques}[thm]{Question}

\newtheorem{rmk}[thm]{Remark}

\theoremstyle{remark}

\newtheorem*{claim*}{Claim}

\newcommand{\bP}{\mathbb{P}}

\newcommand{\cF}{\mathcal{F}}

\newcommand{\cO}{\mathcal{O}}

\newcommand{\cX}{\mathcal{X}}

\newcommand{\fa}{\mathfrak{a}}
\newcommand{\fn}{\mathfrak{n}}
\newcommand{\fm}{\mathfrak{m}}

\newcommand{\cRHom}{\operatorname{R\mathcal{H}om}}

\newcommand{\CC}{\mathbb{C}}

\newcommand{\ZZ}{\mathbb{Z}}
\newcommand{\NN}{\mathbb{N}}
\newcommand{\QQ}{\mathbb{Q}}
\newcommand{\RR}{\mathbb{R}}

\newcommand{\fp}{\mathfrak{p}}

 \DeclareMathOperator{\Tor}{Tor}
\DeclareMathOperator{\Ext}{Ext}
\DeclareMathOperator{\cExt}{\mathcal{E}xt}
\DeclareMathOperator{\Hom}{Hom}

\DeclareMathOperator{\CI}{\textnormal{CI-dim}}

\DeclareMathOperator{\Tr}{\textnormal{Tr}}
\DeclareMathOperator{\hh}{H}

\DeclareMathOperator{\G}{G}

\DeclareMathOperator{\Gdim}{\textnormal{G-dim}}
\DeclareMathOperator{\len}{\textup{length}}

\DeclareMathOperator{\Supp}{Supp}
\DeclareMathOperator{\Spec}{Spec}

\DeclareMathOperator{\Proj}{Proj}

\DeclareMathOperator{\pd}{pd}
\DeclareMathOperator{\fd}{fd}
\DeclareMathOperator{\height}{height}
\DeclareMathOperator{\cx}{cx}

\DeclareMathOperator{\cod}{codim}
\DeclareMathOperator{\grade}{grade}
\DeclareMathOperator{\edim}{embdim}
\DeclareMathOperator{\depth}{depth}
\DeclareMathOperator{\coker}{coker}

\DeclareMathOperator{\HH}{H}

\DeclareMathOperator{\md}{\operatorname{\mathsf{mod}}}
\DeclareMathOperator{\cohe}{\operatorname{\mathsf{coh}}}
\DeclareMathOperator{\reg}{reg}

\newcommand{\ul}{\underline}

\DeclareMathOperator{\db}{\operatorname{\mathsf{D}^b}}
\DeclareMathOperator{\dpf}{\operatorname{\mathsf{D}^{perf}}}
\DeclareMathOperator{\sg}{\operatorname{\mathsf{D}_{sg}}}

\def\Tr{\mathsf{Tr}\hspace{0.01in}}

\def\urltilda{\kern -.15em\lower .7ex\hbox{\~{}}\kern
	.04em}\def\urldot{\kern -.10em.\kern -.10em}\def\urlhttp{http\kern
	-.10em\lower -.1ex\hbox{:}\kern -.12em\lower 0ex\hbox{/}\kern
	-.18em\lower 0ex\hbox{/}} 

\newcommand{\bb}{\left[ \begin{smallmatrix}}
	\newcommand{\eb}{\end{smallmatrix} \right]}

\begin{document}
	
	\title[Remarks on a conjecture of Huneke and Wiegand and the vanishing of (co)homology]{Remarks on a conjecture of Huneke and Wiegand \\and the vanishing of (co)homology}
	
	
	\author[Olgur Celikbas]{Olgur Celikbas}
	\address{Olgur Celikbas\\
		Department of Mathematics \\
		West Virginia University\\
		Morgantown, WV 26506-6310, U.S.A}
	\email{olgur.celikbas@math.wvu.edu}
	
	\author[Uyen Le]{Uyen Le}
	\address{Uyen Le\\
		Department of Mathematics \\
		West Virginia University\\
		Morgantown, WV 26506-6310, U.S.A}
	\email{hle1@mix.wvu.edu}
	
	\author[Hiroki Matsui]{Hiroki Matsui}
	\address{Hiroki Matsui\\ Department of Mathematical Sciences,
		Faculty of Science and Technology,
		Tokushima University,
		2-1 Minamijosanjima-cho, Tokushima 770-8506, JAPAN}
	\email{hmatsui@tokushima-u.ac.jp}
	
	\author[Arash Sadeghi]{Arash Sadeghi}
	\address{Arash Sadeghi\\ Dokhaniat 49179-66686, Gorgan, IRAN}
	\email{sadeghiarash61@gmail.com}
	
	\thanks{Celikbas was partly supported by WVU Mathematics Excellence and Research Funds (MERF). Matsui was partly supported by JSPS Grant-in-Aid for JSPS Fellows 19J00158.}
	
	\maketitle{}

	\begin{abstract} In this paper we study a long-standing conjecture of Huneke and Wiegand which is concerned with the torsion submodule of certain tensor products of modules over one-dimensional local domains. We utilize Hochster's theta invariant and show that the conjecture is true for two periodic modules. We also make use of a result of Orlov and formulate a new condition which, if true over hypersurface rings, forces the conjecture of Huneke and Wiegand to be true over complete intersection rings of arbitrary codimension. Along the way we investigate the interaction between the vanishing of Tate (co)homology and torsion in tensor products of modules, and obtain new  results that are of independent interest.
	\end{abstract}

	\section{Introduction} Throughout $R$ denotes a commutative Noetherian local ring with unique maximal ideal $\fm$ and residue field $k$, and all $R$-modules are assumed to be finitely generated.
	
	In commutative algebra there are several questions about tensor products of modules that are notoriously difficult to solve; see, for example, \cite{D2013}. A fine example of such a question is the following long-standing conjecture of Huneke and Wiegand:
	
	\begin{conj} \label{HWC} (Huneke - Wiegand; see \cite[page 473]{HW1}) \label{impossible} Let $R$ be a one-dimensional local ring and let $M$ be a nonfree and torsion-free $R$-module. Assume $M$ has rank (e.g., $R$ is a domain). Then the torsion submodule of $M\otimes_{R}M^{\ast}$ is nonzero, i.e., $M\otimes_{R}M^{\ast}$ has (nonzero) torsion, where $M^{\ast}=\Hom_R(M,R)$.
	\end{conj}
	
	Conjecture \ref{HWC}, over Gorenstein rings, is in fact a special case of a celebrated conjecture of Auslander and Reiten \cite{AuRe} and is wide open in general, even for two generated ideals over complete intersection domains of codimension two; see Remark \ref{rmk}. On the other hand, besides some other special cases, Conjecture \ref{HWC} is known to be true over hypersurface rings; see \cite{HW1} and \cite{HW3} for the details. In fact, since maximal Cohen-Macaulay modules (that have no free direct summand) are two-periodic over hypersurface rings, the hypersurface case of Conjecture \ref{HWC} is subsumed by the following result:
	
	\begin{thm} \label{thmintro0} (\cite[4.17]{Ce}) Let $R$ be a one-dimensional local complete intersection domain and let $M$ be a nonzero $R$-module. Assume $M$ is two-periodic, i.e., $M \cong \Omega_R^2M$. Then $M\otimes_RM^{\ast}$ has torsion.
	\end{thm}
	
	In this paper we study Conjecture \ref{HWC} for the case where $R$ is a domain. Our aim concerning Theorem \ref{thmintro0} is twofold. In section 2 we give a proof of Theorem \ref{thmintro0} -- which is entirely different from the one given in \cite{Ce} -- by using Tate (co)homology; see Corollary \ref{tc5}. This approach motivates us to seek, and hence obtain, new results about the vanishing of Tate (co)homology and torsion in tensor products that are independent of Conjecture \ref{HWC}; see, for example, Corollary \ref{tc1} and Theorem \ref{tp1}. Furthermore, we generalize Theorem \ref{thmintro0} in section 3. More precisely, we remove the complete intersection hypothesis from Theorem \ref{thmintro0} and hence prove:
	
	\begin{thm} \label{propint} Let $R$ be a one-dimensional local domain and let $M$ be a nonzero $R$-module that is two-periodic. Then $M\otimes_RM^{\ast}$ has torsion.
	\end{thm}
	
	The proof of Theorem \ref{propint} relies upon an invariant, referred to as the Hochster's theta invariant, and is established in the paragraph following Proposition \ref{lemth1}. In section 3, motivated by Theorems \ref{thmintro0} and \ref{propint}, we also provide an example of a two-periodic module over a one-dimensional local ring that is not a complete intersection; see Example \ref{exGP}.
	
	A remarkable theorem of Orlov \cite{O} determines an equivalence between the singularity category of $R$ and that of $\Proj A$ for the generic hypersurface $A$ of $R$. The gist of our work in Section 4 is to exploit Orlov's theorem and show that Conjecture \ref{impossible} holds over all one-dimensional complete intersection domains in case a certain condition we formulate holds for all hypersurface domains. Our approach to use Orlov's theorem to attack Conjecture \ref{HWC} seems to be new and it establishes the following theorem; see the paragraph following Remark \ref{before the proof}.
	\begin{thm} \label{thmintro} Conjecture \ref{HWC} is true over each one-dimensional local complete intersection domain provided that the following condition holds: 
		
		Whenever $R$ is a local hypersurface domain, $c$ is a positive integer, $M$ and $M \otimes_R \Ext_R^{c-1}(M, R)$ are Cohen-Macaulay $R$-modules, both of which have grade $c-1$, it follows $M$ has finite projective dimension.
	\end{thm}
	
	Note that the case where $c=1$ of the condition stated in Theorem \ref{thmintro} is nothing but the condition of Conjecture \ref{HWC}. We do not know whether or not each hypersurface domain satisfies the condition stated in Theorem \ref{thmintro}, but we are now able to translate the problem of Conjecture \ref{HWC} to a problem over hypersurface rings; the new advantage we have is that homological algebra is better understood over hypersurface rings than over complete intersection rings. 
	
	
	\section{A proof of Theorem \ref{thmintro0} via Tate homology}
	
	In this section we use Tate (co)homology and give a proof of Theorem \ref{thmintro0} that is distinct from the one obtained in \cite{Ce}. Along the way we obtain general results that should be useful to further understand the vanishing of Tate (co)homology and torsion; see, for example, Corollary \ref{tc0} and Proposition \ref{tp1}. 
	
	We start by recording several preliminary results some of which will also be used in this section. For the definitions and basic properties of homological dimensions, such as the Gorenstein dimension $\Gdim$ and the complete intersection dimension $\CI$, we refer the reader to \cite{AuBr, AvBu, AGP}.
	
	\begin{chunk} \label{syz} Let $R$ be a local ring, $M$ be an $R$-module and let $n$ be an integer. If $n>0$, then $\Omega^{n}_RM$ denotes the $n$th \emph{syzygy} of $M$, that is, the image of the $n$th differential map in a minimal free resolution of $M$. Also, if $M$ is totally reflexive and $n<0$, then $\Omega^{n}_RM$ denotes the $n$th \emph{cosyzygy} of $M$, that is, the image of the $R$-dual of the $n$th differential map in a minimal free resolution of $M^{\ast}$. Note, by convention, we have $\Omega^{0}_R(M)=M$.
	\end{chunk}
	
	\begin{chunk} Let $R$ be a local ring and let $M$ be an $R$-module. Then the Auslander \emph{transpose} of $M$, denoted by $\Tr M$, is the cokernel of the map $f^{\ast} = \Hom_R(f,R)$, where $F_1 \stackrel{f}
		{\longrightarrow} F_0 \to M \to 0$ is part of the minimal free resolution of $M$; see \cite{AuBr}. Note that $M^{\ast} \cong \Omega^2_R \Tr M$. Note also that $\Tr M$ is unique, up to isomorphism, since so are minimal free resolutions.
	\end{chunk}
	
	The following fact is used for \ref{EPO} and Propositon \ref{lemth1},
	
	\begin{chunk} \label{Tobs} If $R$ is a local ring and $M$ is an $R$-module such that $\Tor_1^R(M, \Tr M)=0$, then $M$ is free; see \cite[A1]{AuGo} and also \cite[3.9]{Yo}.
	\end{chunk}  
	
	Next we recall the definitions of Tate homology and cohomology. Although their definitions do not require the ring to be local, we keep the local setting for simplicity.
	\begin{chunk}  \label{Tate} 
		Let $R$ be a local ring and let $M$ be an $R$-module. A complex $\mathbf{T}$ of free $R$-modules is said to be \emph{totally acyclic} provided that $\hh_n(\mathbf{T})=0=\hh_n(\Hom_R(\mathbf{T},R))$ for all $n\in\ZZ$. A \emph{complete resolution} of $M$ is a diagram
		$\mathbf{T} \overset{\vartheta}{\longrightarrow} \mathbf{P} \overset{\pi}{\longrightarrow}M$,
		where $\mathbf{P}$ is a projective resolution, $\mathbf{T}$ is a totally
		acyclic complex and $\vartheta$ is a morphism of complexes such that $\vartheta_i$ is an isomorphism
		for all $i \gg 0$. 
		
		Assume $\mathbf{T}\rightarrow \mathbf{P}\rightarrow M$ is a complete resolution of $M$. Then, for an $R$-module $N$ and for $i \in \ZZ$, the \emph{Tate homology} $\widehat{\Tor}_i^R(M,N)$ and the \emph{Tate cohomology} $\widehat{\Ext}_{R}^{i}(M,N)$ of $M$ and $N$ over $R$ are defined as $\hh_i(\mathbf{T}\otimes_RN)$ and $\hh^i(\Hom_R(\mathbf{T},N))$, respectively. 
		
		It is known that $\Gdim_R(M)<\infty$  if and only if $M$ has a complete resolution; see \cite[3.1]{AM}. Hence $\widehat{\Tor}_i^R(M,N)$ and $\widehat{\Ext}_{R}^{i}(M,N)$ are defined for each $R$-module $N$ in case $\Gdim_R(M)<\infty$.
		\pushQED{\qed} 
		\qedhere
		\popQED	
	\end{chunk}
	
	The following are some of the fundamental properties of Tate (co)homology modules; see, for example \cite{AvBu, AM, CJDF} or \cite[2.11 and 2.12]{ArCe}.
	
	\begin{chunk} \label{TateP} Let $R$ be a local ring, and let $M$ and $N$ be $R$-modules such that $\Gdim_R(M)<\infty$. 
		\begin{enumerate}[\rm(i)]
			\item If $i>\Gdim_R(M)$, then it follows $\widehat{\Tor}_i^R(M,N)\cong\Tor_i^R(M,N)$ and $\widehat{\Ext}^i_R(M,N)\cong\Ext^i_R(M,N)$.
			\item If $0\to M' \to M \to M'' \to 0$ is a short exact sequence of $R$-modules, where $\Gdim_R(M')<\infty$ or $\Gdim_R(M'')<\infty$, then, for each $i\in \ZZ$, there is an exact sequence of cohomology of the form:
			$$\widehat{\Ext}_{R}^i(M'',N) \to \widehat{\Ext}_{R}^i(M,N) \to \widehat{\Ext}_{R}^i(M',N) \to  \widehat{\Ext}_{R}^{i+1}(M'',N),$$
			and also of homology of the form:
			$$\widehat{\Tor}_{i}^R(M',N) \to \widehat{\Tor}_{i}^R(M,N) \to \widehat{\Tor}_{i}^R(M'',N) \to  \widehat{\Tor}_{i-1}^R(M',N).$$
			\item $\widehat{\Tor}_{i+n}^R(M,N) \cong \widehat{\Tor}_{i}^R(\Omega^n M,N)$ and $\widehat{\Ext}^{i+n}_R(M,N) \cong \widehat{\Ext}^{i}_R(\Omega^n M,N)$ for all $i, n\in \ZZ$.
			\item If $\pd_R(M)<\infty$, then $\widehat{\Ext}^{i}_R(M,N)=0=\widehat{\Tor}_i^R(M,N)$ for all $i\in\ZZ$.
			\item If $M$ is totally reflexive, then $\widehat{\Tor}_i^R(M,N)\cong\widehat{\Ext}^{-i-1}_R(M^*,N)\cong \widehat{\Ext}^{-i+1}_R(\Tr M,N)$ for all $i\in\mathbb{Z}$.
		\end{enumerate}
	\end{chunk}
	
	In the following we collect some results that are used in the proof of Theorem \ref{tt1}. 
	
	\begin{chunk}\label{MCM} \label{JoF} Let $R$ be a local ring, and let $M$ and $N$ be $R$-modules such that $n\geq 1$. 
		\begin{enumerate}[\rm(i)]
			\item If $\depth_R(M\otimes_RN)\geq n$, $\depth_R(N)\geq n-1$ and $\len_R(\Ext^i_R(\Tr M,N))<\infty$ for all $i=1, \ldots, n$, then it follows that $\Ext^i_R(\Tr M,N)=0$ for all $i=1, \ldots, n$; see \cite[4.2(ii)]{ArCe}.
			\item If $\Ext^i_R(\Tr M,N)=0$ for all $i=1, \ldots, n$, then it follows that $\depth_R(M\otimes_RN)\geq\min\{n,\depth_R(N)\}$; see \cite[4.2(iii)]{ArCe}.
			\item If $\pd_R(M)\leq \depth_R(N)$ and $\len_R(\Tor_i^R(M,N))<\infty$ for all $i\geq 1$, then it follows $\Tor_i^R(M,N)=0$ for all $i\geq 1$; see \cite[2.2]{JAB}.
			\item If $\pd_R(M)<\infty$ and $\Tor_i^R(M,N)=0$ for all $i\geq 1$, then the depth formula for $M$ and $N$ holds, that is, $\depth_R(M\otimes_RN)=\depth_R(N)-\pd_R(M)$; see \cite[1.2]{Au}. \qed
		\end{enumerate}	
	\end{chunk}
	
	The following approximation results are classical; see \cite[1.1]{AuBu} and also \cite[3.1 and 3.3]{CIA}.
	
	\begin{chunk}\label{app} Let $R$ be a local ring and let $M$ be an $R$-module such that $\Gdim_R(M)=n<\infty$.
		\begin{enumerate}[\rm(i)]
			\item There is a short exact sequence of $R$-modules $0\rightarrow M\rightarrow X\rightarrow G\rightarrow0$, where $\pd_R(X)=n$ and $G$ is totally reflexive. Such an exact sequence is called a \emph{finite projective hull} of $M$.
			\item There is a short exact sequence of $R$-modules $0 \rightarrow Y \rightarrow X \rightarrow M \rightarrow 0$, where $\pd_R(Y)=n-1$ and $X$ is totally reflexive. Such an exact sequence is called a \emph{Cohen-Macaulay approximation} of $M$. \qed
			
		\end{enumerate}
	\end{chunk}
	
	
	The next result can be deduced from \cite[7.1(a)]{ACS}, which is a more general result. Here we provide a different and self-contained proof for the convenience of the reader.
	
	\begin{prop}\label{tt1} Let $R$ be a local ring and let $M$ and $N$ be $R$-modules. Assume $n\geq 1$ is an integer and the following conditions hold: 
		\begin{enumerate}[\rm(i)]
			\item $\len_R(\Tor_i^R(M,N))<\infty$ for all $i\geq 1$.
			\item $\len_R(\widehat{\Tor}_i^R(M,N))<\infty$ for all $i\in\ZZ$.
			\item $\Gdim_R(M)\leq \depth_R(N)-n$.
		\end{enumerate}
		Then $\widehat{\Tor}_i^R(M,N)=0$ for all $i=-n+1, \ldots, 0$ if and only if $\depth_R(M\otimes_RN)\geq n$.
	\end{prop}
	
	\begin{proof} We start by considering a finite projective hull of $M$, that is, a short exact sequence of $R$-modules
		\begin{equation}\tag{\ref{tt1}.1}
			0\rightarrow M\rightarrow X\rightarrow G\rightarrow0,
		\end{equation}
		where $\pd_R(X)=\Gdim_R(M)$ and $G$ is totally reflexive; see \ref{app}(i). Then, by \ref{TateP}(v), the following holds for all $i\in \ZZ$:
		\begin{equation}\tag{\ref{tt1}.2}
			\widehat{\Tor}_i^R(G,N)\cong\widehat{\Ext}^{-i-1}_R(G^*,N)\cong\widehat{\Ext}^{-i+1}_R(\Tr G,N).
		\end{equation}
		It follows, since $\pd_R(X)<\infty$, that $\widehat{\Tor}_i^R(X,N)=0$ for all $i\in \ZZ$; see \ref{TateP}(iv). So, in view of \ref{TateP}(ii), the short exact sequence in (\ref{tt1}.1) yields the following isomorphisms for all $i\in \ZZ$:
		\begin{equation}\tag{\ref{tt1}.3}
			\widehat{\Tor}_i^R(M,N)\cong\widehat{\Tor}_{i+1}^R(G,N).
		\end{equation}
		Note that, as $\len_R(\widehat{\Tor}_i^R(M,N))<\infty$ for all $i\in\ZZ$, it follows from (\ref{tt1}.2) and (\ref{tt1}.3) that both $\len_R(\widehat{\Tor}_i^R(G,N))$ and $\len_R(\widehat{\Ext}^{i}_R(\Tr G,N))$ are finite for each $i\in\ZZ$. Therefore, since $G$ is totally reflexive, \ref{TateP}(i) shows, for each $i\geq 1$, that:
		\begin{equation}\tag{\ref{tt1}.4}
			\len(\Tor_i^R(G,N))<\infty  \text{ and } \len_R(\Ext^{i}_R(\Tr G,N))<\infty.
		\end{equation}
		
		Recall that we assume $\len_R(\Tor_i^R(M,N))<\infty$ for all $i\geq 1$. So, by (\ref{tt1}.1) and (\ref{tt1}.4), it follows that $\len_R(\Tor_i^R(X,N))<\infty$ for all $i\geq 1$. As  $\pd_R(X)=\Gdim_R(M)\leq \depth_R(N)-n<\depth_R(N)$, we conclude by \ref{JoF}(iii) that $\Tor_i^R(X,N)=0$ for all $i\geq 1$. Hence, by tensoring (\ref{tt1}.1) with $N$, we obtain the following exact sequence:
		\begin{equation}\tag{\ref{tt1}.5}
			0\rightarrow\Tor_1^R(G,N)\rightarrow M\otimes_RN\rightarrow X\otimes_RN\rightarrow G\otimes_RN\rightarrow0.
		\end{equation}
		Furthermore, the depth formula \ref{MCM}(iv), in view of the vanishing of $\Tor_i^R(X,N)$ for all $i\geq 1$, shows:
		\begin{equation}\tag{\ref{tt1}.6}
			\depth_R(X\otimes_RN)=\depth_R(N)-\Gdim_R(M)\geq n.
		\end{equation}
		
		Now we assume $\depth_R(M\otimes_RN)\geq n$ and proceed to prove $\widehat{\Tor}_i^R(M,N)=0$ for all $i=-n+1, \ldots, 0$. Note $\len(\Tor_1^R(G,N))<\infty$; see (\ref{tt1}.4). Therefore, (\ref{tt1}.5) shows that $\Tor_1^R(G,N)=0$ and also $\depth_R(G\otimes_RN)\geq n-1$ since both $\depth_R(X\otimes_RN)$ and $\depth_R(M\otimes_RN)$ are at least $n$; see (\ref{tt1}.6). We know from (\ref{tt1}.4) that $\len_R(\Ext^{i}_R(\Tr G,N))<\infty$ for all $i=1, \ldots, n-1$. Hence \ref{MCM}(i) implies that $\widehat{\Ext}^i_R(\Tr G,N)\cong \Ext^{i}_R(\Tr G,N)=0$ for all $i=1, \ldots, n-1$. Consequently, (\ref{tt1}.2) and (\ref{tt1}.3) yield the required vanishing of Tate Tor modules.
		
		Next we assume $\widehat{\Tor}_i^R(M,N)=0$ for all $i=-n+1, \ldots, 0$ and proceed to prove $\depth_R(M\otimes_RN)\geq n$. Note $\Tor_1^R(G,N)\cong \widehat{\Tor}_1^R(G,N)\cong \widehat{\Tor}_0^R(M,N)$ and $\Ext^i_R(\Tr G,N)\cong\widehat{\Ext}^i_R(\Tr G,N)\cong \widehat{\Tor}_{-i}^R(M,N)$ for all $i=1, \ldots, n-1$; see  (\ref{tt1}.2) and (\ref{tt1}.3). Therefore, we conclude $\Tor_1^R(G,N)=0=\Ext^i_R(\Tr G,N)$ for all $i=1, \ldots, n-1$.
		Now, since $\depth_R(N)\geq n$ and $\Ext^i_R(\Tr G,N)=0$ for all $i=1, \ldots, n-1$, we use \ref{MCM}(ii) and deduce that $\depth_R(G\otimes_RN)\geq \min\{n-1,\depth_R(N)\}\geq n-1$. Recall that $\depth_R(X\otimes_RN)\geq n$; see (\ref{tt1}.6). Thus, since $\Tor_1^R(G,N)=0$, we obtain, by the depth lemma applied to the exact sequence (\ref{tt1}.5), that $\depth_R(M\otimes_RN)\geq n$.
	\end{proof}
	

	
	\begin{chunk} \label{EPO} Let $R$ be a local ring and let $M$ and $N$ be $R$-modules. Assume $\CI_R(M)<\infty$ or $\CI_R(N)<\infty$. Then the following conditions are equivalent; see \cite[4.9]{AvBu}.
		\begin{enumerate}[\rm(i)]
			\item $\Tor_i^R(M,N)=0$ for all $i\gg 0$.
			\item $\Tor_i^R(M,N)=0$ for all $i>\CI_R(M)$.
			\item $\widehat{\Tor}_i^R(M,N)=0$ for all $i\in \ZZ$.
		\end{enumerate}
		
		Moreover, if $\CI_R(M)=0$ and $N=M^{\ast}$, then $M$ is free if and only if one of the above equivalent conditions holds: this is because, if $\CI_R(M)=0$ and $\widehat{\Tor}_i^R(M,M^{\ast})=0$ for all $i\in \ZZ$, then, since $M^{\ast}\cong \Omega_R^{2} \Tr M$, it follows that $\Tor_1^R(M, \Tr M)=0$, which forces $M$ to be free; see \ref{Tobs} and \ref{TateP}(i).
		
		We should also note, if $\CI_R(M)<\infty$, then the conditions (i), (ii) and (iii) stated above are also equivalent for Ext and Tate cohomology modules; see \cite[4.7]{AvBu}.
	\end{chunk}
	
	\begin{cor}\label{tc0}
		Let $R$ be a $d$-dimensional Gorenstein local ring and let $M$ and $N$ be maximal Cohen-Macaulay $R$-modules such that $\pd_{R_{\fp}}(M_{\fp})<\infty$ for all $\fp \in \Spec(R)-\{\fm\}$. Then $M\otimes_{R}N$ is maximal Cohen-Macaulay if and only if $\widehat{\Tor}_i^R(M,N)=0$ for all $i=-d+1, \ldots, 0$.
	\end{cor}
	
	\begin{proof} Let $\fp \in \Spec(R)-\{\fm\}$. Then it follows that $\CI_{R_{\fp}}(M_{\fp})=0$. Therefore, by \ref{EPO}, we have $\Tor_i^R(M,N)_{\fp}=0$ and $\widehat{\Tor}_j^R(M,N)_{\fp}=0$ for all $i\geq 1$ and for all $j\in \ZZ$. Now the claim follows from Theorem \ref{tt1}.
	\end{proof}
	
	
	In passing we record an immediate consequence of Corollary \ref{tc0}; it yields a new characterization of torsionfreeness of tensor products in terms of the vanishing of Tate homology:
	
	\begin{cor}\label{tc1}
		Let $R$ be a one-dimensional Gorenstein local domain and let $M$ and $N$ be torsion-free $R$-modules. Then $M\otimes_{R}N$ is torsion-free if and only if $\widehat{\Tor}_0^R(M,N)=0$.
	\end{cor}
	
	Let us note that, in view of Corollary \ref{tc1}, Conjecture \ref{HWC} can be stated over Gorenstein rings as follows: if $R$ is a one-dimensional Gorenstein domain and $M$ is a torsion-free $R$-module such that $\widehat{\Tor}_0^R(M,M^{\ast})=0$, then $M$ is free.

	Our main result in this section is Theorem \ref{tp1} which yields a new criterion for the vanishing of Tate (co)homology. More precisely, Theorem \ref{tp1} is an extension of \cite[4.8]{CeD}, which considers the vanishing of (absolute) cohomology under the same hypotheses. To prove the theorem, we record a few more preliminary definitions and results.
	
	\begin{chunk} \label{cx} Let $R$ be a local ring and let $M$ be an $R$-module. Then the \emph{complexity} $\cx_R(M)$ of $M$ is defined as
		$\inf\{r \in \NN \cup \{0\}  : A \in \RR  \text{ such that } \dim_k(\Ext^n_R(M,k))\leq A \cdot n^{r-1} \text{ for all } n \gg 0\}$; see \cite{Av1}.
		
		Note, it follows from the definition that, $\cx_R(M)=0$ if and only if $\pd_R(M)<\infty$, and $\cx_R(M)\leq 1$ if and only if $M$ has bounded Betti numbers. Furthermore, the following properties hold:
		\begin{enumerate}[\rm(i)]
			\item If $\CI_R(M)<\infty$, then it follows that $\cx_R(M)\leq \edim(R)-\depth(R)$; see \cite[5.6]{AGP}. \pushQED{\qed} 
			\item If $\CI_R(M)=0$, then it follows $\CI_R(M^{\ast})=0$ and $\cx_R(M)=\cx_R(M^{\ast})$; see \cite[4.2]{BeJ3}. \qed
		\end{enumerate}
	\end{chunk}

	
	
	
	
	
	
	In the following $\overline{\G}(R)_{\QQ}$ denotes the \emph{reduced Grothendieck group with rational coefficients}, that is, $\overline{\G}(R)_{\QQ}=(\G(R)/\ZZ \cdot [R])\otimes_{\ZZ}\QQ$, where $\G(R)$ is the Grothendieck group of (finitely generated) $R$-modules. Also, $[N]$ denotes the class of a given $R$-module $N$ in $\overline{\G}(R)_{\QQ}$. 
	
	\begin{thm}\label{tp1} Let $R$ be a local ring and let $M$ and $N$ be $R$-modules. Assume:
		\begin{enumerate}[\rm(i)]
			\item $\cx_R(M)=c$.
			\item $\CI_R(M)<\infty$.
			\item $\pd_{R_{\fp}}(M_{\fp})<\infty$ for all $\fp\in\Spec(R)-\{\fm\}$.
			\item $[N]=0$ in $\overline{G}(R)_{\mathbb{Q}}$.
		\end{enumerate}
		Then, given an integer $n$, the following hold:
		\begin{enumerate}[\rm(a)]
			\item If $\widehat{\Ext}^i_R(M,N)=0$ for all $i=n, \ldots, n+c-1$, then $\widehat{\Ext}^i_R(M,N)=0$ for all $i\in\ZZ$.
			\item If $\widehat{\Tor}_i^R(M,N)=0$ for all $i=n, \ldots, n+c-1$, then $\widehat{\Tor}_i^R(M,N)=0$ for all $i\in\ZZ$.
		\end{enumerate}
	\end{thm}
	\begin{proof} 
		
		We proceed and prove the statement in part (a) first. Set $\dim(R)=d$ and $X=\Omega^{n-d-1}_R\Omega^d_R M$. Then  $X$ is totally reflexive and \ref{TateP}(iii) yields the following isomorphisms for all $i \in \ZZ$:
		\begin{equation}\tag{\ref{tp1}.1}
			\widehat{\Ext}^i_R(M,N) \cong \widehat{\Ext}^{i-d}_R(\Omega^d_R M,N) \cong \widehat{\Ext}^{i-n+1}_R(X, N).
		\end{equation}
		
		Therefore, for all $j\geq 1$, we obtain:
		\begin{equation}\tag{\ref{tp1}.2}
			\Ext^{j}_R(X,N) \cong \widehat{\Ext}^{j}_R(X,N) \cong \widehat{\Ext}^{j+n-1}_R(M,N).
		\end{equation}
		Here the first and the second isomorphisms are due to \ref{TateP}(i) and (\ref{tp1}.1), respectively. Hence, (\ref{tp1}.2) and our assumption give:
		\begin{equation}\tag{\ref{tp1}.3}
			\Ext^{j}_R(X,N)=0 \text{ for all } j=1, \ldots, c.
		\end{equation}
		
		Note that $\CI_R(X)=0$ and $\cx_R(X)=\cx_R(M)=c$; see  \cite[1.9.1]{AGP} and \cite[3.6]{Sean}. As $X$ is locally free on the punctured spectrum of $R$, we use \cite[4.8]{CeD} and conclude from (\ref{tp1}.3) that $\Ext^{j}_R(X,N) = 0$ for all $j \ge 1$, that is, $\Ext^{j}_R(M,N) = 0$ for all $j \gg 0$. Now the required vanishing follows from \ref{EPO}.
		
		Next we prove part (b). Note \ref{TateP}(iii) yields the following isomorphisms for all $i \in \ZZ$:
		\begin{equation}\tag{\ref{tp1}.4}
			\widehat{\Tor}_i^R(M,N) \cong \widehat{\Tor}_{i-d}^R(\Omega^d_R M,N) \cong \widehat{\Tor}_{i-n+1}^R(X,N).
		\end{equation}
		Then our assumption and (\ref{tp1}.4) yield:
		\begin{equation}\tag{\ref{tp1}.5}
			\widehat{\Tor}_j^R(X,N) = 0 \text{ for all } j=1, \ldots, c.
		\end{equation}
		
		Note, as $X$ is totally reflexive, \ref{TateP}(v) gives the following isomorphism for all $i\in \ZZ$:
		\begin{equation}\tag{\ref{tp1}.6}
			\widehat{\Ext}_R^{-j-1}(X^{\ast},N) \cong \widehat{\Tor}_j^R(X,N). 
		\end{equation}
		
		Now (\ref{tp1}.5) and (\ref{tp1}.6) show that:
		\begin{equation}\tag{\ref{tp1}.7}
			\widehat{\Ext}_R^{-j-1}(X^{\ast},N)=0 \text{ for all } j=1, \ldots, c. 
		\end{equation}
		
		It follows, since $\CI_R(X) =0$, that $\CI_R(X^*) = \CI_R(X) =0$ and $\cx_R(X^*)=\cx_R(X)=c$; see \ref{cx}(ii). Now 
		we use part (a) of the theorem with (\ref{tp1}.7) and conclude that $\widehat{\Ext}_R^{-j-1}(X^{\ast},N)=0$ for all $j \in \ZZ$. Then (\ref{tp1}.4) and (\ref{tp1}.6) establish the required vanishing of Tate homology.
	\end{proof}

	We proceed and collect several examples of rings for which hypothesis (iv) of Theorem \ref{tp1} holds; see \cite[2.5, 2.6 and 4.11]{CeD} and also Example \ref{exteta}.
	
	\begin{chunk} \label{GV} Let $R$ be a local ring and let $N$ be an $R$-module. Then $[N]=0$ in $\overline{\G}(R)_{\QQ}$ for each one of the following cases:
		\begin{enumerate}[\rm(i)]
			\item $N=\Omega^n_R(M)$ for some $n\geq 0$ and for some $R$-module $M$ such that $\len_R(M)<\infty$.
			\item $R$ is Artinian.
			\item $R$ has dimension one and $N$ has rank.
			\item $R$ is a one-dimensional domain.
			\item $R$ is a two-dimensional normal domain with torsion class group. 
			\item $R$ is a two-dimensional complete rational singularity such that $k$ is algebraically closed of characteristic zero, or $k$ is finite of positive characteristic, or $k$ is the algebraic closure of a finite field which has positive characteristic.
		\end{enumerate}
	\end{chunk}

	\begin{cor}  \label{tc3} Let $R$ be a local ring of positive depth and let $M$ be a nonfree $R$-module. Assume the following conditions hold:
		\begin{enumerate}[\rm(i)]
			\item $\CI_R(M)=0$ and $\cx(M)\leq 1$.
			\item $\len_R(\Tor_i^R(M,M^{\ast}))<\infty$ for all $i\gg 0$.
			\item $[M]$ or $[M^{\ast}]$ is zero in $\overline{G}(R)_{\mathbb{Q}}$.
		\end{enumerate}
		Then it follows that $\depth_R(M\otimes_{R}M^{\ast})=0$.
	\end{cor}
	
	\begin{proof} We start by noting that $\cx_R(M)=1$: this is because we assume $M$ is a nonfree $R$-module such that $\depth(M)=\depth(R)$ and $\cx_R(M)\leq 1$. Note also we have that $\CI_R(M^{\ast})=0$ and $\cx_R(M^{\ast})=1$; see \ref{cx}(ii). Furthermore, if $\fp \in \Spec(R)-\{\fm\}$, then our assumption (ii) and \ref{EPO} imply that $M_{\fp}$ is free over $R_{\fp}$. In other words, $M$ is locally free on the punctured spectrum of $R$.
		
		
		Now suppose $\depth_R(M\otimes_{R}M^{\ast})\geq 1$. Then, setting $n=1$, we conclude from Theorem \ref{tt1} that $\widehat{\Tor}_0^R(M,M^{\ast})=0$. So Theorem \ref{tp1} yields the vanishing of
		$\widehat{\Tor}_i^R(M,M^{\ast})$ for all $i\in\ZZ$. As we know $\CI_R(M)=0$, this implies that $M$ is free; see \ref{EPO}. Therefore, the depth of $M\otimes_{R}M^{\ast}$ must be zero.
	\end{proof}
	
	We should note that, concerning the hypothesis in part (iii) of Corollary \ref{tc3}, we are not aware of an example of a ring $R$ and an $R$-module $M$ such that $\CI_R(M)=0$ and $[M^{\ast}]\neq [M]=0$ in $\overline{G}(R)_{\mathbb{Q}}$. However, we observe next that such an example cannot occur over Gorenstein rings. 
	
	\begin{rmk} Let $R$ be a Gorenstein local ring. It is well-known that the Grothendieck group $\G(R)$ is generated by maximal Cohen-Macaulay $R$-modules. In fact, given an $R$-module $M$ with minimal a free resolution $\mathbf{F}=(F_i)$, it follows that $[M] = \sum_{i=0}^{d-1} (-1)^{i} [F_i] + (-1)^{d} [\Omega^d M]$, where $d$ is the dimension of $R$.
		
		Now we define a map $\Phi: \G(R) \to \G(R)$ as follows:
		$$
		\Phi([M]) = \sum_{i=0}^d (-1)^{i-1} [\Ext_R^i(M, R)].
		$$ 
		One can check that $\Phi$ is a well-defined group homomorphism. Let $M$ be a maximal Cohen-Macaulay $R$-module. As $\Ext^i_R(M,R)=0$ for all $i\geq 1$, it follows that $\Phi([M]) = [M^*]$ and so $\Phi^2([M]) = [M^{**}] = [M]$. This shows that $\Phi$ is an isomorphism as $\G(R)$ is generated by maximal Cohen-Macaulay $R$-modules. 
		Since $\Phi$ sends $[R]$ to $[R]$, it induces an isomorphism $\overline{\Phi}: \overline{G}(R)_{\mathbb{Q}} \to \overline{G}(R)_{\mathbb{Q}}$. Hence, if  
		$M$ is a maximal Cohen-Macaulay $R$-module, then  $[M] = 0$ in $\overline{G}(R)_{\mathbb{Q}}$ if and only if $\overline{\Phi}([M])=[M^*]=0$ in $\overline{G}(R)_{\mathbb{Q}}$. \qed
	\end{rmk}
	
	If $R$ is a one-dimensional local ring and $M$ is a torsion-free $R$-module such that  $\CI_R(M)<\infty$ and $\cx_R(M)\leq 1$ (for example, $R$ is a hypersurface ring, or $R$ is a complete intersection ring and $M$ has bounded Betti numbers), then it follows that $M\cong \Omega_R^{2}M$ (for our purpose we may assume $M$ has no free direct summand); see \cite[7.3]{AGP} for the details. Hence Theorem \ref{thmintro0} is subsumed by the next result which was proved in \cite{Ce} via different techniques. 
	
	\begin{cor} \label{tc5} (\cite[4.10]{OC2021}) Let $R$ be a one-dimensional local ring and let $M$ be a nonfree torsion-free $R$-module. Assume $M$ has rank (e.g., $R$ is a domain). Assume further $\CI_R(M)<\infty$ and  $\cx_R(M)\leq 1$. Then it follows that $M\otimes_{R}M^{\ast}$ has torsion.
	\end{cor}
	
	\begin{proof} It follows, as $R$ is a one-dimensional local ring and $M$ has rank, that $[M]=0$ in $\overline{G}(R)_{\mathbb{Q}}$; see \ref{GV}. Hence the claim follows from Corollary \ref{tc3}.
	\end{proof}
	
	Note that, in the next section we will prove Theorem \ref{propint} and hence generalize Corollary \ref{tc5}. We finish this section by pointing out some flexibility about the hypotheses of Corollary \ref{tc5}.
	
	\begin{rmk} Let $R$ be a one-dimensional local domain and let $M$ be a nonzero torsion-free $R$-module. If $\CI_R(M)<\infty$ or $\CI_R(M^{\ast})<\infty$, then it follows that 
		$\CI_R(M)=0=\CI_R(M^{\ast})$; see \ref{cx}(ii) and \cite[4.5(i)]{OC2021}. So,  if $\CI_R(M)<\infty$ or $\CI_R(M^{\ast})<\infty$, and $\cx_R(M)\leq 1$ or $\cx_R(M^{\ast})\leq 1$ 
		(for example, if $\CI_R(M^{\ast})<\infty$ and $\cx_R(M)\leq 1$, or $\CI_R(M)<\infty$ and $\cx_R(M^{\ast})\leq 1$), then it follows that $\CI_R(M)<\infty$ and $\cx_R(M)\leq 1$.
	\end{rmk}

	
	
	

	
	\section{Conjecture \ref{HWC} for two-periodic modules over one-dimensional domains}
	
	In this section we give a proof of Theorem \ref{propint}; see the paragraph following Proposition \ref{lemth1}. The primary result of the section is Theorem \ref{th2} which relies upon a slightly modified version of the theta invariant.
	This invariant was initially defined by Hochster \cite{Ho1} to study the direct summand conjecture; it was further developed by Dao, for example, to study the vanishing of Tor; see \cite{Da2, Da3, Da1}.
	
	We start by recalling the definition of the theta invariant; see also \cite[section 2]{Da3}.
	
	\begin{dfn} (\cite{Ho1}) \label{th1} Let $R$ be a local ring and let $M$ be an $R$-module. Assume the following hold:
		\begin{enumerate}[\rm(i)]
			\item $M\cong \Omega^{2}_RM$.
			\item $\pd_{R_{\fp}}(M_{\fp})<\infty$ for all $\fp \in \Spec(R)-\{\fm\}$.
		\end{enumerate}
		Then, for each $i\geq 1$, it follows that $\Tor_{i}^R(M,N) \cong \Tor_{i+2}^R(M,N)$ and $\len_R(\Tor_{i}^R(M,N))<\infty$. Therefore, given an $R$-module $N$, the \emph{theta invariant} for the pair $(M,N)$ is defined as follows:
		\begin{equation}\tag{\ref{th1}.1}
			\theta^R(M,N)=\len_R\big(\Tor_{2n}^R(M,N)\big)-\len_R\big(\Tor_{2n-1}^R(M,N)\big) \text{ for some } n\geq 1.
		\end{equation}
		Note that $\theta^R(M,N)$ is well-defined, i.e., its value is independent of the integer $n$ used in the definition. 
	\end{dfn}

	Next we show that the theta function is additive on short exact sequences.
	
	\begin{thm} \label{th2} Let $R$ be a local ring and let $M$ be an $R$-module. Assume the following hold:
		\begin{enumerate}[\rm(i)]
			\item $M\cong \Omega_R^{2}M$.
			\item $\pd_{R_{\fp}}(M_{\fp})<\infty$ for all $\fp \in \Spec(R)-\{\fm\}$.
			\item $0 \to X \stackrel{f}{\longrightarrow} Y \stackrel{g}{\longrightarrow}  Z \to 0$ is a short exact sequence of $R$-modules.
		\end{enumerate}
		Then it follows that $\theta^R(M,Y)=\theta^R(M,X)+\theta^R(M,Z)$. 
	\end{thm}
	
	\begin{proof} Note that, since $M\cong \Omega^{2}_R M$, we can pick a minimal free resolution $F_{\bullet}=(F_{n}, \partial_n)$ of $M$ such that $F_{n}=F_{n+2}$ for each $n\geq 0$ and $\partial_n=\partial_{n+2}$ for all $n\geq 1$. 
		
		Fix an integer $n\geq 1$. Then, by tensoring $f$ with $F_{n}$, we get the following commutative diagram:
		\begin{equation}\notag{}
			\xymatrixcolsep{3pc}\xymatrix{
				F_{n}\otimes_RX \ar[r]^{1_{F_{n}}\otimes f} \ar@{^{}=}[d]_{} & F_n\otimes_RY   \ar@{^{}=}[d]^{} &   \\ 
				F_{n+2}\otimes_RX \ar[r]^{1_{F_{n+2}}\otimes f} & F_{n+2}\otimes_RY} 
		\end{equation}
		
		The diagram above yields the following commutative diagram on homologies:
		\begin{equation}\notag{}
			\xymatrixcolsep{4pc}\xymatrix{
				\Tor_n^R(M,X)  \ar[r]^{\Tor_n^R(M,f)} \ar@{^{}->}[d]_{\cong} & \Tor_n^R(M,Y)    \ar@{^{}->}[d]^{\cong} &   \\
				\Tor_{n+2}^R(M,X) \ar[r]^{\Tor_{n+2}^R(M,f)} & \Tor_{n+2}^R(M,Y)}
		\end{equation}
		
		Now set $X_n=\ker\big(\Tor_n^R(M,f)\big)$. Then it follows, since the diagram above involving the Tor modules, is commutative with vertical maps being isomorphisms, that $X_n \cong X_{n+2}$.
		
		Note that the short exact sequence $0 \to X \stackrel{f}{\longrightarrow} Y \stackrel{g}{\longrightarrow}  Z \to 0$ gives rise to the following exact sequence:
		\begin{align}\tag{\ref{th2}.1}
			0 \to X_{n+2} \to \Tor^R_{n+2}(M, X) \to \Tor^R_{n+2}(M, Y) \to  \Tor^R_{n+2}(M, Z) \to &\\
			\notag{} \Tor^R_{n+1}(M, X) \to   \Tor^R_{n+2}(M, Y) \to  \Tor^R_{n+1}(M, Z) \to X_{n} \to 0.
		\end{align}
		As we have $X_n\cong X_{n+2}$, by taking the alternating sum of lengths of the terms in the exact sequence (\ref{th2}.1), we conclude that $\theta^R(M,Y)=\theta^R(M,X)+\theta^R(M,Z)$. 
	\end{proof}
	
	Recall that $\overline{\G}(R)_{\QQ}$ denotes the reduced Grothendieck group with rational coefficients.
	
	\begin{cor} \label{corth1} Let $R$ be a local ring and let $M$ be an $R$-module. Assume the following hold:
		\begin{enumerate}[\rm(i)]
			\item $M\cong \Omega_R^{2}M$.
			\item $\pd_{R_{\fp}}(M_{\fp})<\infty$ for all $\fp \in \Spec(R)-\{\fm\}$.
		\end{enumerate}
		Then $\theta^R(M,-):\overline{G}(R)_{\QQ} \to \QQ$ is a well-defined function.
	\end{cor}
	
	\begin{proof} It is clear from Theorem \ref{th2} that $\theta^R(M,-)$ is a well-defined function on $\G(R)$. As $\theta^R(M,R)=0$, the result follows; see Definition \ref{th1}.
	\end{proof}
	
	\begin{chunk}\label{rigid} If $M$ and $N$ are $R$-modules, then the pair $(M,N)$ is said to be $n$-Tor-rigid for some $n\geq 1$ provided that the following condition holds: if $\Tor_i^R(M,N)=0$ for all $i=r, r+1, \ldots, r+n-1$ for some $r\geq 1$, then $\Tor_j^R(M,N)=0$ for all $j\geq r$. 
		
		An $R$-module $M$ is called \emph{Tor-rigid} if $(M,N)$ is $1$-Tor-rigid for each $R$-module $N$. For example, if $R$ is a regular local ring, or a hypersurface which is quotient of an unramified regular local ring, then each $R$-module of finite projective dimension is Tor-rigid; see \cite[2.2]{Au} and \cite[Cor. 1 and Thm. 3]{Li}. \qed
	\end{chunk}
	
	
	\begin{cor} \label{corth2} Let $R$ be a local ring and let $M$ and $N$ be $R$-modules. Assume:
		\begin{enumerate}[\rm(i)]
			\item $\pd_{R_{\fp}}(M_{\fp})<\infty$ for all $\fp \in \Spec(R)-\{\fm\}$.
			\item $[N]=0$ in $\overline{\G}(R)_{\QQ}$.
			\item $M\cong \Omega_R^{q}M$ for some even integer $q\geq 2$.
		\end{enumerate}
		Then the pair $(M,N)$ is $(q-1)$-Tor-rigid.
	\end{cor}
	
	\begin{proof} Set $X=M \oplus \Omega^2_R(M) \oplus \cdots \oplus \Omega^{q-2}_R(M)$. Then, since $M\cong \Omega_R^{q}M$, it follows that $X\cong \Omega_R^{2}X$. Note also, if $r\geq 1$, then by the definition of $X$, we have: 
		\begin{equation}\tag{\ref{corth2}.1}
			\Tor_r^R(X,N)=0 \text{ if and only if } \Tor_{r+2i}^R(M,N)=0 \text{ for all } i=0, \ldots, (q/2)-1.
		\end{equation}
		
		Next, to show $(M,N)$ is $(q-1)$-Tor-rigid, suppose $\Tor_i^R(M,N)=0$ for all $i=n, n+1, \ldots, n+q-2$ for some $n\geq 1$. Then, by (\ref{corth2}.1), it follows that $\Tor_n^R(X,N)=0$. Hence, if we show $\Tor_{n+1}^R(X,N)=0$, then, as  $X\cong \Omega_R^{2}X$, we conclude that $\Tor_i^R(X,N)=0$ for all $i\geq 1$; this yields the vanishing of $\Tor_j^R(M,N)$ for all $j\geq 1$ and establishes that $(M,N)$ is $(q-1)$-Tor-rigid.
		
		Notice $\pd_{R_{\fp}}(X_{\fp})<\infty$, for all $\fp \in \Spec(R)-\{\fm\}$. So, Corollary \ref{corth1} shows $\theta^R(X,-):\overline{G}(R)_{\QQ} \to \QQ$ is well-defined. Thus $\theta^R(X,N)=0$ because $[N]=0$ in $\overline{\G}(R)_{\QQ}$. Consequently, as we know $\Tor_n^R(X,N)$ vanishes, we deduce by the definition of the theta function that $\Tor^R_{n+1}(X,N)=0$; see Definition \ref{th1}.
	\end{proof}
	
	The following example points out that hypothesis (ii) of Corollary \ref{corth2} is needed, even over isolated hypersurface singularities. 
	
	\begin{eg} \label{exteta} Let $R=\CC[\![x,y, z, w]\!]/(xw-yz)$, $M=R/(x,z)$ and $N=R/(x,y)$. Then one can check that $R$ is a three-dimensional $(A_1)$ singularity (and hence is a domain) and $\Tor_2^R(M,N)=0\neq \Tor_3^R(M,N)$. Thus we have $\theta^R(M,N)=-1$. This implies that $[N]\neq 0$ in $\overline{\G}(R)_{\QQ}$. Let us also note, since $\overline{\G}(S)_{\QQ}\cong \QQ$ for each one-dimensional $(A_1)$ singularity $S$, it follows from Kn\"{o}rrer periodicity that $\overline{\G}(R)_{\QQ}\cong \QQ$; see \cite[12.10 and 13.10]{Yo}. \qed
	\end{eg}
	
	As Example \ref{exteta} shows, the conclusion of Corollary \ref{corth2} may fail if hypothesis (ii) is not satisfied. However, even without this assumption, the corollary can be useful to produce Tor-rigid modules:
	
	\begin{cor} Let $R$ be a local ring, $M$ be an $R$-module, and let $x\in \fm$ be a non zero-divisor on $R$. Assume that the following hold:
		\begin{enumerate}[\rm(i)]
			\item $\pd_{R_{\fp}}(M_{\fp})<\infty$ for all $\fp \in \Spec(R)-\{\fm\}$.
			\item $M\cong \Omega_R^{2}M$.
		\end{enumerate}
		Then the pair $(M/xM,N/xN)$ is Tor-rigid over $R/xR$ for each torsion-free $R$-module $N$. Therefore, the pair $(M/xM,M/xM)$ is Tor-rigid over $R/xR$.
	\end{cor}
	
	\begin{proof} Let $N$ be a torsion-free $R$-module. As there is an exact sequence $0\to N \stackrel{x}{\rightarrow} N \to N/xN \to 0$, it follows that $[N/xN]=[N]-[N]=0$ 
		in $\overline{\G}(R)_{\QQ}$. Hence Corollary \ref{corth2} shows that the pair $(M,N/xN)$ is Tor-rigid over $R$. Now, for an integer $n\geq 0$, it follows 
		$\Tor_n^{R/xR}(M/xM,N/xN)=0$ if and only if $\Tor_n^{R}(M,N/xN)=0$ if and only if $\Tor_i^{R/xR}(M/xM,N/xN)=0$ for all $i\geq n$; see \cite[Lemma 2, page 140]{Mat}.
	\end{proof}
	
	
	
	In the following we record some classes of rings over which periodic modules have certain Tor-rigidity property; see also Corollary \ref{gozlemle} for a characterization of Tor-rigid modules over one-dimensional Gorenstein domains.
	
	\begin{cor} \label{corth4} Let $R$ be a local ring. Assume that one of the following holds:
		\begin{enumerate}[\rm(i)]
			\item $R$ is Artinian.
			\item $R$ is a one-dimensional domain.
			\item $R$ is a two-dimensional normal domain with torsion class group.
		\end{enumerate}
		If $M$ is an $R$-module such that $M\cong \Omega_R^{q}M$ for some even integer $q\geq 2$, then $M$ is $(q-1)$-Tor-rigid. Therefore, if $M\cong \Omega_R^{2}M$, then $M$ is Tor-rigid. 
	\end{cor}
	
	\begin{proof} It is known that, if $R$ is as in part (i), (ii) or (iii), then $\overline{G}(R)_{\QQ}=0$; see \ref{GV}. Therefore the result follows from Corollary \ref{corth2}.\end{proof}   
	
	\begin{prop} \label{lemth1} Let $R$ be a local ring and let $M$ be a nonzero $R$-module. Assume:
		\begin{enumerate}[\rm(i)]
			\item $M_{\fp}$ is free for each associated prime ideal of $R$.
			\item $M\cong \Omega_R^{q}M$ for some even integer $q\geq 2$.
			\item $M$ is Tor-rigid.
		\end{enumerate}
		Then $M\otimes_RM^{\ast}$ has torsion.
	\end{prop}
	
	\begin{proof} Note that $M^{\ast} \neq 0$ since $M$ is nonzero and torsion-free. Hence there exists a short exact sequence of $R$-modules $0 \to M^{\ast} \to F \to C \to 0$, where $F$ is free. Tensoring this exact sequence with $M$, we see that $\Tor^R_1(C,M) \hookrightarrow M\otimes_RM^{\ast}$. Note that $\Tor^R_1(C,M)$ is a torsion module since $M_{\fp}$ is free for each associated prime ideal of $R$.
		
		Now assume $M\otimes_RM^{\ast}$ is torsion-free. Then $\Tor^R_1(C,M)$, and hence $\Tor^R_i(C,M)$ for each $i\geq 1$, vanishes. This yields $\Tor_i^R(M,M^{\ast})=0$ for all $i\geq 1$. As $M^{\ast} \cong \Omega^2_R \Tr M$ (up to free summands), we conclude $\Tor_1^R(M, \Tr M) \cong \Tor_1^R(\Omega^q_R M, \Tr_R M) \cong \Tor_1^R(M, \Omega^q_R \Tr M) \cong \Tor_{1+q-2}^R(M, M^{\ast})=0$. This implies $M$ is free and hence $M=0$ since $M\cong \Omega_R^{q}M$; see \ref{Tobs}. Therefore, $M\otimes_RM^{\ast}$ must have torsion.
	\end{proof}
	
	We are now ready to prove Theorem \ref{propint} advertised in the introduction. Recall that our argument removes the complete intersection hypothesis from Theorem \ref{thmintro0}.
	
	\begin{proof}[Proof of Theorem \ref{propint}] Note that the module $M$ considered in the theorem is Tor-rigid; see Corollary \ref{corth4}. Therefore, $M\otimes_RM^{\ast}$ has torsion by Proposition \ref{lemth1}.
	\end{proof}
	
	\begin{rmk} \label{rmk} Let $R$ be a one-dimensional local domain and let $M$ be a torsion-free $R$-module. 
		\begin{enumerate}[\rm(i)]
			\item If $R$ is Gorenstein, then it follows that $M\otimes_RM^{\ast}$ is torsion-free if and only if $\Ext^1_R(M,M)=0$; see \cite[5.9]{HJ}. So Conjecture \ref{HWC} predicts that the celebrated conjecture of Auslander and Reiten is true over one-dimensional Gorenstein rings, even if only one Ext module vanishes; see \cite{AuRe} and \cite{HJ} for the details. However, if $R$ is not Gorenstein, then torsionfreeness and the vanishing of Ext may not be equivalent: for example, if $R$ is Cohen-Macaulay with a canonical module $\omega \ncong R$ and if $R$ has minimal multiplicity, then it follows that $\Ext^1_R(\omega,\omega)=0$, but $\omega \otimes_R \omega^{\ast}$ has torsion; see \cite[3.6]{HW3}. 
			
			\item It seems interesting that, even if $R$ is not Gorenstein, when $M\cong \Omega_R^{2}M$, it still follows that $M\otimes_RM^{\ast}$ is torsion-free if and only if $\Ext^1_R(M,M)=0$, since either of these two conditions forces $M$ to be zero; this follows from Theorem \ref{propint} and the fact that, when $M$ is a Tor-rigid module over a local ring $R$ and $\Ext^n_R(M,M)=0$ for some $n\geq 0$, then $\pd_R(M)<n$; see, for example, \cite[3.1.2]{D2013}. \qed
		\end{enumerate}
	\end{rmk}
	
	In general there exist two-periodic modules that do not have finite complete intersection dimension. We build on an example of Gasharov and Peeva and construct such a module in the next example (note the ring considered in the example is not a complete intersection); cf. Theorems \ref{thmintro0} and \ref{propint}.
	
	\begin{eg} (\cite[3.10]{GPE}) \label{exGP} Let $k$ be a field and fix $\alpha \in k$ such that $\alpha^4=1\neq \alpha^3$. Let $R= k[\![x_1, x_2, x_3, x_4]\!]/I_\alpha$, where $I_\alpha$ is the ideal of $R$ generated by the elements $x_1^2-x_2^2,\,\,\, x_3^2,\,\,\, x_4^2,\,\,\,  x_3x_4 ,\,\,\, x_1x_4+ x_2x_4,\,\,\, \alpha \cdot x_1x_3 + x_2x_3$.
		Then the complex
		$$
		\cdots \xrightarrow{\partial_3} R^{\oplus 2} \xrightarrow{\partial_2} R^{\oplus 2} \xrightarrow{\partial_1} R^{\oplus 2} \to N \to 0
		$$
		is exact, where
		$$
		N=\coker(\partial_1) \text{ and } \partial_n = 
		\begin{pmatrix}
			x_1 & \alpha^n \cdot x_3 + x_4 \\
			0 & x_2 \\
		\end{pmatrix} \text{ for each } n\geq 1.
		$$
		
		It follows $\Omega_R^4 N \cong \coker(\partial_5) = \coker(\partial_1) \cong N$ and that $N \ncong \Omega_R^i N $ for each $i=1, 2, 3$.
		Then, by setting $X=N\oplus \Omega^2_R N$, we see that $X \cong \Omega^2_R X$ and hence $\cx_R(X)\leq 1$. Moreover, we have $\CI_R(X)=\infty$: otherwise we would have $\CI_R(N)<\infty$ and so $N \cong \Omega^2_R N$ since $\cx_R(N)\leq 1$.
		
		Next we set $T= R[\![t]\!]$ and $M= X\otimes_{R}T$. Then, since $T$ is a faithfully flat extension of $R$, we conclude that $T$ is a one-dimensional Gorenstein ring, $M \cong \Omega_T^2 M$, and also $\CI_T(M)=\infty$. \qed
	\end{eg}
	
	The ring $T$ in Example \ref{exGP} is not a domain and also the $T$-module $M$ does not have rank. This raises the following question which, in view of Conjecture \ref{HWC} and Theorem \ref{propint}, should be important:
	
	\begin{ques} \label{u} Is there a one-dimensional local ring $T$ and a $T$-module $M$ such that $M$ has rank over $T$ (e.g., $T$ is a domain), $M \cong \Omega_T^2 M$, and $\CI_T(M)=\infty$? \qed
	\end{ques}
	
	
	We finish this section by recording some further observations concerning the torsion submodule of tensor products of the form $M\otimes_RM^{\ast}$. Let us point out that the conclusion of Theorem \ref{propint} may fail if the ring in question is not one-dimensional; for example, \cite[3.5]{CeD2} provides such an example of a two-dimensional local hypersurface domain. In the aforementioned example, the module considered is torsion-free but it is not maximal Cohen-Macaulay. On the other hand, when maximal Cohen-Macaulay modules are considered, there is a partial result, which we recall next:
	
	\begin{chunk} \label{DO} If $R$ is an even-dimensional local hypersurface and $M$ is a nonfree maximal Cohen-Macaulay $R$-module that is locally free on the punctured spectrum of $R$, then $M\otimes_RM^{\ast}$ has torsion; see \cite[3.7]{CeD2}.
	\end{chunk}
	
	Note that the module $M$ in \ref{DO} is two-periodic since it is maximal Cohen-Macaulay over a hypersurface ring.  As we are concerned with periodic modules of even period, we proceed and investigate whether there is an extension of \ref{DO} for such modules over rings that are not necessarily hypersurfaces. For that we first prove:
	
	\begin{prop} \label{iyi} Let $R$ be a $d$-dimensional Cohen-Macaulay local ring with canonical module $\omega$ and let $M$ be a nonzero $R$-module that is locally free on the punctured spectrum of $R$. Assume $d=2nr$ and $M\cong \Omega_R^{2n} M$ for some positive integers $n$ and $r$. Then $M\otimes_RM^{\dagger}$ has torsion, where $M^{\dagger}=\Hom_R(M, \omega)$.
	\end{prop}
	
	\begin{proof} Note that, if $M\otimes_RM^{\dagger}$ is torsion-free, then $M$ is zero due to the following isomorphisms:
		\begin{align}\notag{}
			0=\HH_{\fm}^{0}(M\otimes_RM^{\dagger})^{\vee} & \cong \Ext^d_R(M\otimes_RM^{\dagger},\omega) & \\ & \notag{} \cong \Ext^d_R(M,M) & \\ & \cong  \notag{}\Ext^{2n}_R\big(\Omega_R^{2n(r-1)} M, M\big) & \\ & \cong  \notag{} \Ext^{2n}_R(M,M) & \\ & \notag{} \notag{} \cong \Ext^{1}_R(\Omega_R^{2n-1} M,\Omega_R^{2n} M). 
		\end{align}
		Here the first isomorphism follows by the local duality \cite[3.5.9]{BH} and the second is due to \cite[2.3]{GT68}.
	\end{proof}
	The next corollary of Proposition \ref{iyi} yields an extension of \ref{DO}. 
	
	\begin{cor} \label{corson} Let $R$ be an even-dimensional Gorenstein local ring and let $M$ be a nonzero $R$-module that is locally free on the punctured spectrum of $R$. If  $M\cong \Omega_R^{2} M$, then $M\otimes_RM^{\ast}$ has torsion.
	\end{cor}
	
	We should note that Corollary \ref{corson} may fail if the dimension of $R$ is odd; see \cite[3.12]{CeD2}.

	
	\section{A condition implying Conjecture \ref{HWC} over complete intersection rings}
	In this section each complete intersection ring $R$ of codimension $c$ is of the form $S/(\underline{x})$ for some regular local ring $(S, \fn)$ and for some $S$-regular sequence $\underline{x}=x_1, \ldots, x_c\subseteq \fn^2$. Such a complete intersection ring is called a hypersurface when $c=1$.
	
	As mentioned in the introduction, although Conjecture \ref{HWC} is true over hypersurface rings, it is wide open for complete intersection rings that have codimension at least two. The aim of this section is to formulate a condition over hypersurfaces, which, if true, forces Conjecture \ref{HWC} to be true over all complete intersection rings.

	
	We start with a setup and then recall a theorem of Orlov \cite[Section 2]{O} which plays a key role for our argument. For the basic definitions that are not defined in this section, we refer the reader to \cite{Hbook}. 
	
	\begin{chunk} \label{g1} Throughout, given a commutative Noetherian ring $R$  and a scheme $T$, $\md R$ and $\cohe T$ denote the category of (finitely generated) $R$-modules and the category of coherent $\cO_T$-modules, respectively.
		It follows that there is a category equivalence: $\md R \cong \cohe (\Spec (R))$ given by $M \mapsto \widetilde{M}$, see, for example, \cite[II 5.5]{Hbook}.
	\end{chunk}
	
	\begin{chunk} Let $X$ be a scheme. A \emph{perfect complex} on $X$ is a bounded complex of coherent sheaves on $X$ which has finite flat dimension as a complex. 
		
		The {\it singularity category} is defined as the quotient:
		$$
		\sg(X) = \db(\cohe X)/\dpf(X),
		$$ 
		where $\dpf(X)$ denotes the full subcategory of the bounded derived category $\db(\cohe X)$ consisting of all perfect complexes on $X$. 
		
		Note that a perfect complex on $\Spec(R)$ is isomorphic in $\db(\md R)$ to a bounded complex of finitely generated projective modules via the equivalence $\md R \cong \cohe (\Spec (R))$ mentioned in \ref{g1}. Moreover, if $R$ is Gorenstein, then the singularity category of $R$ is equivalent to the stable category of maximal Cohen-Macaulay $R$-modules by \cite[4.4.1]{Bu}.
		
	\end{chunk}
	
	\begin{chunk} \label{setup} Let $R$ be a complete intersection ring of codimension $c$ and let 
		$\underline{t}=t_1, \ldots, t_c$ be indeterminates over $S$. Then we define a graded hypersurface ring
		$$
		\displaystyle{A= S[\underline{t}]/ \bigg(\sum_{i=1}^c x_i t_i\bigg)}
		$$
		where the grading is given by $\deg s = 0$ for all $s \in S$, and $\deg t_i = 1$ for each $i$.
		
		Now we consider the natural surjections:
		$$
		S[\underline{t}] \twoheadrightarrow A \twoheadrightarrow A/(\underline{x}) = R[\underline{t}].
		$$
		These surjections yield the following commutative diagram of schemes:
		$$
		\xymatrix{
			Z:= \bP^{c-1}_R \ar@{^{(}-_>}[r]^i \ar[d]_p & Y:= \Proj(A) \ar@{^{(}-_>}[r]^u & X:= \bP^{c-1}_S \ar[d]^q \\
			\Spec(R) \ar@{^{(}-_>}[rr]^j && \Spec(S).
		}
		$$
		In the above diagram, the morphisms $i$, $u$, and $j$ are closed immersions that are induced by the surjections $A \twoheadrightarrow R[\ul{t}]$, $S[\ul{t}] \twoheadrightarrow A$, $S \twoheadrightarrow R$, respectively. Also, the morphisms $p$ and $q$ are canonical. We note that $i$ is a {\it regular closed immersion} of codimension $c-1$, that is, the ideal sheaf of $i$ is locally generated by a regular sequence of length $c-1$. Furthermore, the morphism $p$ is flat.
		
		We consider two functors $p^*: \md R \to \cohe Z$ and $i_*: \cohe Z \to \cohe Y$, which are defined as follows:
		\begin{enumerate}[\rm(i)]
			\item $p^*: \md R \cong \cohe (\Spec(R)) \to \cohe Z$ is the pullback along $p$, where $p^*(M)$ is the $\cO_Z$-module $\widetilde{M[\underline{t}]}$ associated with a graded $R[\ul{t}]$-module $M[\underline{t}]= M \otimes_R R[\underline{t}] \cong (M \otimes_R A)/x(M \otimes_R A)$; see \cite[p116, Definition]{Hbook}. 
			\item $i_*: \cohe Z \to \cohe Y$ is the pushout along $i$. Note that, every object of $\cohe Z$ is isomorphic to $\widetilde{L}$ for some graded $R[\ul{t}]$-module $L$. Then $i_*(\widetilde{L})$ is isomrphic to $\widetilde{L_A}$, where $L_A$ is the graded $R[\ul{t}]$-module $L$ considered as a graded $A$-module via the ring map $A \twoheadrightarrow  R[\underline{t}]$. 
		\end{enumerate}
		
		Notice, since $p$ is flat and $i$ is closed immersion, it follows that $p^{\ast}$ and $i_*$ are exact functors, see \cite[02N4 and 01QY]{Sta}. Therefore, by deriving these functors, we also obtain triangle functors:
		\begin{equation}\notag{}
			i_*: \db(\cohe Z) \to \db(\cohe Y) \text{ and } p^*: \db(\md R) \to \db(\cohe Z).	
		\end{equation}
		Here, the functors are given by applying $p^*$ and $i_*$ component-wise.
		\pushQED{\qed} 
		\qedhere 
		\popQED
	\end{chunk}
	
	\begin{chunk} Note that $Y= \Proj A$ is an integral scheme and hence every ring of section is an integral domain.	
		Indeed, since $S[t_1, \ldots, t_c]$ is an integral domain, we can easily check that $\sum_{i=1}^c x_it_i$ is an irreducible element.
		Therefore, it is a prime element as $S[t_1, \ldots, t_c]$ is a UFD and hence $A$ is a domain; see \ref{setup}. \qed
	\end{chunk}
	
	Throughout this section we keep the notations and the setting of \ref{setup}. The following result of Orlov \cite{O} plays a key role in the proof Theorem \ref{thmmain}.
	
	\begin{chunk} (\cite[2.1]{O}, see also \cite[A.4]{BW}) \label{Orl} The triangle functor $\Phi=i_* p^*:  \db(\md R) \to \db(\cohe Y)$ induces a triangle equivalence $\overline{\Phi}:  \sg(R) \xrightarrow{\cong} \sg(Y)$.  
		
		Note, $\Phi(M) \cong \widetilde{(M \otimes_R A) / \underline{x} (M \otimes_R A)}$ for any $R$-module $M$.
		Moreover, we have 
		\begin{align*}
			\pd_R M < \infty \Longleftrightarrow M\cong0 \text{ in } \sg(R) \Longleftrightarrow \Phi(M)\cong0 \text{ in } \sg(Y) \\
			\qquad \Longleftrightarrow \fd_{\cO_Y} \Phi(M)<\infty \Longleftrightarrow \pd_{\cO_{Y, y}}\Phi(M)_y<\infty \mbox{ for all } y \in Y.
		\end{align*}
		Here, the first and third equivalences follow by the definition of singularity categories, the second equivalence follows from Orlov's theorem \ref{Orl}, and \cite[III 9.2(e)]{Hbook} proves the last equivalence. 
	\end{chunk}
	
	\begin{chunk}\label{tensor} Let $M$ and $N$ be $R$-modules.  Then we have the following:
		\begin{align*}
			\Phi(M) \otimes_{\cO_Y} \Phi(N) &= \widetilde{\big((M \otimes_R A) / \underline{x} (M \otimes_R A)\big)} \otimes_{\cO_Y} \widetilde{\big((N \otimes_R A) / \underline{x} (N \otimes_R A)\big)} 	\\
			& \cong \widetilde{\bigg(\big[(M \otimes_R A) / \underline{x} (M \otimes_R A)\big] \otimes_{A} \big[(N \otimes_R A) / \underline{x} (N \otimes_R A))\big]\bigg)} \\
			&= \widetilde{\bigg(\big[(M \otimes_R N) \otimes_R A\big] / \underline{x}\big[(M \otimes_R N) \otimes_R A\big]\bigg)} \\
			&= \Phi(M \otimes_R N). \pushQED{\qed} 
		\end{align*}
		where the isomorphism follows from \cite[Proof of II 5.12(b)]{Hbook}. \qed
	\end{chunk}
	
	Next we proceed to determine $\Phi(M^*)$. For this, we use the Grothendieck duality theorem \cite{RD}.
	
	\begin{lem} \label{gd} Let  $ \cF \in \cohe Z$ and $n \in \ZZ$. Then there is a natural isomorphism as follows:	
		$$
		\cExt_{\cO_Y}^n(i_* \cF, \cO_Y) \cong i_* \cExt_{\cO_Z}^{n-c+1}(\cF, \cO_Z)(1).
		$$
	\end{lem}
	
	\begin{proof} Note, by the Grothendieck duality theorem \cite[III 6.7]{RD}, there is a natural isomorphism
		$$
		\cRHom_{\cO_Y}(i_* \cF, \cO_Y) \cong i_* \cRHom_{\cO_Z}(\cF, i^!\cO_Y),
		$$
		where $i^!: \db(\cohe Y) \to \db(\cohe Z)$ denotes the right adjoint functor of $i_*: \db(\cohe Z) \to \db(\cohe Y)$. We proceed to prove that $i^! \cO_Y$ is isomorphic to $\cO_Z(1)[-c+1]$.
		
		Note that, by \cite[III 7.3]{RD}, there is an isomorphism 
		$i^! \cO_Y \cong \omega_{Z/Y}[-c+1]$. Here, $\omega_{Z/Y}$ is the relative canonical sheaf of the regular closed immersion $i: Z \hookrightarrow Y$; see \cite[III \S 1]{RD} for its definition.
		On the other hand, by \cite[III 1.5]{RD}, we have an isomorphism of the form $\omega_{Z/Y} \cong \omega_{Z/X} \otimes_{\cO_Z} (i^* \omega_{Y/X})^\vee$.
		The ideal sheaves of $ui: Z \hookrightarrow X$ and $u: Y \hookrightarrow X$ are globally generated by degree $0$ and $1$ regular sequences, respectively. Therefore, the following isomorphisms hold:
		$$
		\omega_{Z/X} \cong \cO_Z \,\,\ \text{ and }\,\,\, \omega_{Y/X} \cong \cO_Y(-1).
		$$ 
		Hence we conclude that $
		i^! \cO_Y \cong \omega_{Z/Y}[-c+1] \cong \cO_Z \otimes_{\cO_Z} i^*(\cO_Y(-1))^\vee [-c+1] \cong \cO_Z(1)[-c+1]$.
	\end{proof}

	\begin{chunk}\label{corst} 
		Let $M$ be an $R$-module. Then there are natural isomorphisms 
		\begin{align*}
			\cExt_{\cO_Y}^n(\Phi(M), \cO_Y) &= \cExt_{\cO_Y}^n(i_*p^*(M), \cO_Y) \\
			&\cong i_* \cExt_{\cO_Z}^{n-c+1}(p^*(M), \cO_Z)(1) \\
			&\cong \Phi(\Ext_{R}^{n-c+1}(M, R))(1)
		\end{align*}
		Here, the first isomorphism uses \ref{gd} and the last isomorphism is due to \cite[III 6.5]{Hbook} together with the following fact: $p^*(P_{\bullet})$ is a resolution of $p^*(M)$ by locally free sheaves of finite rank for a given resolution $P_{\bullet}$ by finite free module of $M$.
		
		Assume further that $M$ is a totally reflexive $R$-module. Then, by definition, if $n \neq 0$, it follows that $\Ext_R^n(M, R) \cong 0$. Therefore, the above isomorphisms yield:
		$$
		\cExt_{\cO_Y}^{i}(\Phi(M), \cO_Y) \cong 
		\begin{cases}
			0 & (i \neq c-1)\\
			\Phi(M^*)(1) &(i=c-1).	
		\end{cases}
		$$
	\end{chunk}
	
	Next we recall the definition of the codimension of a module:
	
	\begin{chunk} \label{cod} Let $A$ be a local ring and let $N$ be an $A$-module. Then the \emph{codimension} $\cod_A(N)$ of $N$ is defined as the codimension of its support $\Supp_A(N)$ as a closed set in $\Spec(A)$. More precisely, we have 
		\begin{equation}\notag{}
			\cod_A(N)=\inf \{\height_A(\fp) \mid \fp \in \Supp_A(N)\}.
		\end{equation}
	\end{chunk}
	
	In the following we record some properties of the codimension that are needed for our argument; see \cite[2.1.2 and 3.3.10]{BH}. Among those is the fact that the codimension of a module does not change when localizing at a prime ideal in its support. 
	
	\begin{chunk}  \label{cod} Let $A$ be a Cohen-Macaulay local ring and let $N$ be a nonzero $A$-module.
		\begin{enumerate}[\rm(i)]
			\item It follows that $\cod_A(N) = \dim(A)-\dim_A(N) = \grade_A(N)= \inf \{i\in \ZZ \mid \Ext_A^i(N, A) \neq 0\}$, where $\grade_A(N)$ denotes the grade of $N$ over $A$.
			\item Assume $A$ admits a canonical module $\omega_A$. 
			\begin{enumerate}[\rm(a)]
				\item Then $N$ is Cohen-Macaulay of codimension $t$ if and only if $\Ext_A^i(N, \omega_A) = 0$ for $i\neq t$.
				\item If $N$ is Cohen-Macaulay of codimension $t$ and $\fp \in \Supp_A(N)$, then $N_\fp$ is Cohen-Macaulay over $A_\fp$ such that $\cod_{A_{\fp}}(N_{\fp})=t$.
				\item If $N$ is Cohen-Macaulay of codimension $t$, then it follows $\Supp_A(N) = \Supp_A (\Ext_A^t(N, \omega_A))$.
			\end{enumerate}
		\end{enumerate}
	\end{chunk}
	
	We define the following conditions for local domains $R$:
	
	\begin{chunk} \label{cond} Let $c\geq 1$ be an integer. 
		\begin{enumerate}[\rm(i)]
			\item If $M$ and $M \otimes_R \Ext_R^{c-1}(M, R)$ are Cohen-Macaulay $R$-modules of codimension $c-1$, then it follows that $\pd_R(M)<\infty$.
			\item If $M$ and $M\otimes_RM^{\ast}$ are maximal Cohen-Macaulay (i.e., Cohen-Macaulay $R$-modules of codimension $0$), then it follows $M$ is free. \qed
		\end{enumerate}
	\end{chunk}
	
	Next is the main result of this section; recall that we keep the notations and the setting of \ref{setup}.
	
	\begin{thm} \label{thmmain} Let $c\geq 1$ be an integer. If each local hypersurface domain satisfies condition (i) of \ref{cond}, then each local complete intersection domain of codimension $c$ satisfies condition (ii) of \ref{cond}.
	\end{thm}
	
	\begin{rmk} \label{before the proof} Prior to giving a proof for Theorem \ref{thmmain}, we remark that condition (ii) of \ref{cond}, or equivalently condition (i) for the case where $c=1$, is nothing but the condition stated in Conjecture \ref{HWC} for one-dimensional local domains. Recall that each hypersurface local domain satisfies condition (ii) of \ref{cond}; see \cite[3.1]{HW1}.
	\end{rmk}
	
	
	\begin{proof}[Proof of Theorem \ref{thmmain}] We assume, for the given integer $c$, that each local hypersurface domain satisfies condition (i) of \ref{cond}. 
		
		Let $R= S/(\underline{x})$ be a domain, where $S$ is a regular local ring and $\ul{x}$ is an $S$-regular sequence of length $c$. Let $M$ be a maximal Cohen-Macaulay $R$-module such that $M \otimes_R M^*$ is maximal Cohen-Macaulay. We proceed to prove that $M$ is free.
		
		First we prove that $\Supp_Y(\Phi(M) ) = \Supp_Y(\cExt_{\cO_Y}^{c-1}(\Phi(M), \cO_Y))$. Fix $y \in \Supp_Y(\Phi(M))$.
		The combination of \ref{corst} and \ref{cod}(ii)(a) show that $\Phi(M)_y$ is a Cohen-Macaulay $\cO_{Y,y}$-module of codimension $c-1$.
		Therefore, by \ref{cod}(ii)(c), we conclude $y \in \Supp_Y(\cExt_{\cO_Y}^{c-1}(\Phi(M), \cO_Y))$ and hence it follows that $\Supp_Y(\Phi(M) ) \subseteqq \Supp_Y(\cExt_{\cO_Y}^{c-1}(\Phi(M), \cO_Y))$. The converse inclusion is trivial.
		Moreover, the support of $\Phi(M)$ equals to the support of $\cX$, where $\cX= \Phi(M) \otimes_{\cO_Y} \cExt_{\cO_Y}^{c-1}(\Phi(M), \cO_Y)(-1)$.
		
		Note, it follows from \ref{tensor} and \ref{corst} that $\Phi(M \otimes_R M^*) \cong  \Phi(M) \otimes_{\cO_Y} \Phi(M^*) \cong \cX$.
		As $M$ and $M \otimes_R M^*$ are totally reflexive $R$-module, by \ref{corst}, we obtain the following:
		\begin{align}\tag{\ref{thmmain}.1}
			&\cExt_{\cO_Y}^i(\cX, \cO_Y) \cong \cExt_{\cO_Y}^i(\Phi(M \otimes_R M^*), \cO_Y) \cong \Phi(\Ext_R^{i-c+1}(M \otimes_R M^*, R))(1) =0 
		\end{align}
		for $i \neq c-1$.

		Let $y \in \Supp_Y(\Phi(M))$. Then the module $\cX_y= \Phi(M)_y \otimes_{\cO_{Y, y}} \Ext_{\cO_{Y, y}}^{c-1}(\Phi(M)_y, \cO_{Y, y})$ is Cohen-Macaulay of codimension $c-1$ over the local hypersurface domain $\cO_{Y, y}$ by \ref{cod}(ii)(a) and (\ref{thmmain}.1). Now our assumption shows that $\Phi(M)_y$ has finite projective dimension over $\cO_{Y, y}$. 
		So, by \ref{Orl}, we have that $\pd_R(M)< \infty$ 
	\end{proof}
	
	Now the proof of Theorem \ref{thmintro} follows as an immediate consequence of Theorem \ref{thmmain}:
	
	\begin{proof}[Proof of Theorem \ref{thmintro}] Let $c\geq 1$ be an integer. If each local hypersurface domain satisfies condition (i) of \ref{cond} (for the given $c$), then Theorem \ref{thmmain} implies that Conjecture \ref{HWC} is true over each one-dimensional local complete intersection domain.
	\end{proof}
	
	\begin{rmk} We note a fact that follows from the proof of Theorem \ref{thmmain}: if each hypersurface domain which is quotient of an equi-characteristic regular local ring satisfies condition (i) of Theorem \ref{thmmain}, then Conjecture \ref{HWC} holds over each one-dimensional complete intersection domain which is quotient of an equi-characteristic regular local ring. \qed
	\end{rmk}
	
	We finish this section by noting that, if we consider Conjecture \ref{HWC} over one-dimensional complete intersection domains that have algebraically closed residue fields, then the proof of 
	Theorem \ref{thmmain} is simplified significantly due to a result in \cite{BJ}:
	
	\begin{rmk}
		Let $R= S/I$ be a one-dimensional local complete intersection domain of codimension $c$ with algebraically closed residue field, and let $M$ be a torsion-free $R$-module such that $I \subseteq \fn^2$ and $M \otimes_R M^*$ is torsion-free. 
		
		Let $f \in I - \fn I$. Then we can extend $\{f\}$ to a minimal generating  set $\{f_1, \ldots, f_c\}$ of $I$, with $f=f_1$, which is necessarily a regular sequence on $S$. Hence $M \otimes_{S/(f_1)} \Ext_{S/(f_1)}^{c-1}(M, S/(f_1)) \cong M \otimes_R M^*$ is a Cohen-Macaulay $S/(f_1)$-module of codimension $c-1$.
		Assuming the condition \ref{cond}(ii), it follows that $\pd_{S/(f_1)}(M)<\infty$. Now \cite[3.3]{BJ} implies that $\pd_R(M)<\infty$.
	\end{rmk}
	
	
	\section*{Appendix: a remark on the rigidity of Tor}
	
	It is known that Tor-rigidity, a subtle property, is a sufficient condition for Conjecture \ref{HWC} to hold over one-dimensional Gorenstein domains; see \ref{rigid} and Remark \ref{rmk}(ii). Motivated by this fact, we examine the vanishing of Tor more closely over Gorenstein rings. The observation we aim to establish in this appendix is the following, which may be helpful for further study Tor-rigidity.
	
	\begin{chunk} \label{gozlemle} Let $R$ be a one-dimensional local Gorenstein domain and let $M$ be an $R$-module. Then the following conditions are equivalent:
		\begin{enumerate}[\rm(i)]
			\item $M$ is Tor-rigid over $R$.
			\item $(M,C)$ is Tor-rigid for each torsion (or equivalently, finite length) $R$-module $C$.
			\item $(M,C)$ is Tor-rigid for each torsion-free  (or equivalently, maximal Cohen-Macaulay) $R$-module $C$.
		\end{enumerate}
	\end{chunk}
	
	We deduce \ref{gozlemle} from the following more general result:
	
	\begin{prop} \label{lemth2} Let $R$ be a $d$-dimensional Gorenstein local ring and let $M$ be an $R$-module. Then the following conditions are equivalent:
		\begin{enumerate}[\rm(i)]
			\item If $C$ is a torsion $R$-module and $\Tor_n^R(M,C)=0$ for some $n\geq d$, then it follows $\Tor_i^R(M,C)=0$ for all $i\geq n$.
			\item If $C$ is a maximal Cohen-Macaulay $R$-module that has rank and $\Tor_n^R(M,C)=0$ for some $n\geq d$, then it follows $\Tor_i^R(M,C)=0$ for all $i\geq n$.
			\item If $C$ is an $R$-module with rank and $\Tor_n^R(M,C)=0$ for some $n\geq d$, then it follows $\Tor_i^R(M,C)=0$ for all $i\geq n$.
		\end{enumerate}
	\end{prop}
	
	\begin{proof} First we show that parts (ii) and (iii) are equivalent, that is, part (ii) implies part (iii). For that assume part (ii) holds. Let $C$ be an $R$-module with rank such that $\Tor_n^R(M,C)=0$ for some $n\geq d$. We want to show that $\Tor_i^R(M,C)=0$ for all $i\geq n$.
		
		We may assume $C$ is not maximal Cohen-Macaulay. Then we consider a Cohen-Macaulay approximation of $C$, that is, a short exact sequence of $R$-modules $0\to L \to X \to C \to 0$, where $\pd_R(L)<\infty$ and $X$ is maximal Cohen-Macaulay; see \ref{app}(ii). Note $\pd_R(L)<d$ so that $\Tor_n^R(M,X)=0$. Thus, as $X$ has rank, it follows from the hypothesis that $\Tor_i^R(M,X)=0$ for all $i\geq n$. This yields $\Tor_i^R(M,C)=0$ for all $i\geq n$, and hence establishes part (iii).
		
		Next we show that part (i) implies (ii). Assume part (i) holds and let $C$ be a maximal Cohen-Macaulay $R$-module with rank such that $\Tor_n^R(M,C)=0$ for some $n\geq d$. We want to show that $\Tor_i^R(M,C)=0$ for all $i\geq n$.
		As $C$ has rank, there exists a short exact sequence $0 \to C \to G \to Y\to 0$, where $G$ is free and $Y$ is torsion; see \cite[1.3]{HW1}. As $\Tor_n^R(M,C)=0$, we have that $\Tor_{n+1}^R(M,Y)=0$; now the hypothesis implies that $\Tor_i^R(M,Y)=0$ for all $i\geq n+1$. Consequently, $\Tor_i^R(M,C)=0$ for all $i\geq n$, as required.
		
		Finally we note a module is torsion if and only if it has rank zero. So part (iii) implies part (i). 
	\end{proof}
	
	\section*{Acknowledgements}
	The authors thank Kenta Sato for pointing to them a simpler proof of Lemma \ref{gd} than the one in a previous version of the manuscript.
	
	\bibliography{a}
	\bibliographystyle{amsplain}
\end{document}